\documentclass[a4paper,11pt]{amsproc}

\title{A number of properties enjoyed by two specially constructed  topologies on $C(X)$ }
\usepackage{amsfonts, amsmath, amssymb, graphicx, mathrsfs, geometry}
\usepackage[all]{xy}
\usepackage{tikz,float}
\usetikzlibrary{shapes.geometric, arrows}
\tikzstyle{arrow} = [thick,->,>=stealth]

\usetikzlibrary{cd, arrows.meta, positioning}
\usepackage{enumitem}
\newdir{ >}{!/6pt/@{ }*:(1,0.2)@_{>}*:(1,-0.2)@^{>}}

\theoremstyle{plain}

\newtheorem{theorem}{Theorem}[section]

\title{A number of properties enjoyed by two specially constructed  topologies on $C(X)$}
\usepackage{amsfonts, amsmath, amssymb, graphicx, mathrsfs}
\usepackage[all]{xy}

\newdir{ >}{!/6pt/@{ }*:(1,0.2)@_{>}*:(1,-0.2)@^{>}}

\theoremstyle{plain}

\newtheorem{lemma}[theorem]{Lemma}

\theoremstyle{definition}
\newtheorem{definition}[theorem]{Definition}

\newtheorem{remark}[theorem]{Remark}
\newtheorem{counter example}[theorem]{Counter Example}

\newtheorem{corollary}[theorem]{Corollary}

\numberwithin{equation}{section}

\DeclareMathOperator{\cl}{cl}

\author[S. Dey]{Soumajit Dey}	\address{Department of Pure Mathematics, University of Calcutta, 35, Ballygunge Circular Road, Kolkata 700019, West Bengal, India}	\email{deysoumajit8@gmail.com}

\author[S. K. Acharyya]{Sudip Kumar Acharyya}	\address{Department of Pure Mathematics, University of Calcutta, 35, Ballygunge Circular Road, Kolkata
	700019, West Bengal, India}	\email{sdpacharyya@gmail.com}

\author[D. Mandal]{Dhananjoy Mandal} \address{Department of Pure Mathematics, University of Calcutta, 35 , Ballygunge Circular Road, Kolkata 700019, West Bengal, India}  \email{dmandal.cu@gmail.com / dmpm@caluniv.ac.in}
\begin{document}
		\title [A number of properties enjoyed by two specially constructed  topologies on $C(X)$]{A number of properties enjoyed by two specially constructed  topologies on $C(X)$ }
			\thanks {The first author extends immense gratitude and thanks to the University Grants Commission, New Delhi, for the award of research fellowship (NTA Ref. No. 211610214962).}
		\maketitle
		
	 \bibliographystyle{plain}
	 \begin{abstract}
	 	If $I$ is an ideal in the ring $C(X)$ of all real valued continuous functions defined over a Tychonoff space $X$, then $X$ is called $I$-$pseudocompact$ if the set $X\setminus \bigcap Z[I]$ is a bounded subset of $X$. Corresponding to $I$, the $m^I$-topology and $u^I$-topology on $C(X)$, akin to that of the well-known $m$-topology and $u$-topology in $C(X)$ respectively are already there in the literature. It is proved amongst others that the $m^I$-topology is first countable if and only if the $u^I$-topology= $m^I$-topology on $C(X)$ if and only if $X$ is $I$-$pseudocompact$. A special case of this result on choosing $I=C(X)$ reads: the $u$-topology and $m$-topology on $C(X)$ coincide if and only if $X$ is pseudocompact. It is established that the $m^I$-topology on $C(X)$ is second countable if and only if it is $\aleph_0$-$bounded$ if and only if $X$ is compact, metrizable and $I=C(X)$. Furthermore it is realized that the $m^I$-topology on $C(X)$ is hemicompact if and only if it is $\sigma$-compact if and only if this topology is $H$-$bounded$ if and only if $X$ is finite and $I=C(X)$. Finally on imposing some natural condition on $I$, it is realized that $X$ is $I$-$pseudocompact$ if and only if $C_{m^I}(X)$ is \v{C}ech complete.
	 \end{abstract}
	 \section{Introduction}
	 The topology of uniform convergence or briefly the $u$-topology and the $m$-topology on $C(X)$ are two fascinating topics for the researchers in the broad area of function spaces. A few pertinent problems addressing a variant of these two topologies on $C(X)$, via an ideal $I$ in $C(X)$ is investigated rather recently. To be specific, let $I$ be an ideal in the ring $C(X)$ of all real valued continuous functions constructed over a Tychonoff space $X$. For $f\in C(X)$, $\epsilon>0$ in $\mathbb{R}$ and $u\in C_+(X)\equiv$ the set of all strictly positive continuous functions on $X$, let $B_u(f,I,\epsilon)=\{g\in C(X): \sup\limits_{x\in X}|f(x)-g(x)|<\epsilon \text{ and }f-g\in I\}$ and $B_m(f,I,u)=\{g\in C(X):|f(x)-g(x)|<u(x) \text{ for all } x\in X\text{ and }f-g\in I\}$. Then the family $\mathcal{B}_u=\{B_u(f,I,\epsilon):f\in C(X),\hspace{.001cm} \epsilon>0\}$ is an open base for some topology designated as the $u^I$-topology on $C(X)$ in \cite{PA2023} and $C(X)$ with the $u^I$-topology is denoted in \cite{PA2023}, by $C_{u^I}(X)$. Incidentally the family $\mathcal{B}_m=\{B_m(f,I,u):f\in C(X),\hspace{.001cm}u\in C_+(X)\}$ makes an open base for a topology called the $m^I$-topology on $C(X)$ in \cite{A2017}. $C(X)$ with the $m^I$-topology is denoted by $C_{m^I}(X)$. The $m^I$-topology on $C(X)$ is initiated in \cite{A2017}, while the $u^I$-topology is introduced in \cite{PA2023}. In both of these two articles, several relevant problems connected with the $m^I$-topology on $C(X)$ are addressed. A number of pertinent problems related to the $u^I$-topology on $C(X)$ are investigated in \cite{PA2023}. We call the space $X$, $I$-$pseudocompact$ if $X\setminus \bigcap\limits_{g\in I}Z(g)=X\setminus \bigcap Z[I]$ is a bounded subset of $X$. With the choice $I=C(X)$, $I$-$pseudocompactness$ of $X$ reduces to $pesudocompactness$ of $X$. In the present article we continue our study on the $u^I$-topology and the $m^I$-topology on $C(X)$, by bringing into focus this time on several aspects of these two topologies, related to metrizability, connectedness, compactness and boundedness. It is already recorded in these two articles that $C_{u^I}(X)$ is a topological group while $C_{m^I}(X)$ is a topological ring and the $m^I$-topology on $C(X)$ is in general finer that the $u^I$-topology. It is established in \cite{PA2023}, Theorem 3.7 that with $I$, a convex ideal in $C(X)$ the $u^I$-topology coincides with the $m^I$-topology if and only if $X\setminus\bigcap\{Z(f):f\in I\}$ is a bounded subset of $X$. A subset $Y$ of $X$ is called bounded if each $f\in C(X)$ is bounded on $Y$. Bounded subsets of $X$ are also called relatively pseudocompact subsets in the literature \cite{m}. In the present article we show that the convexity hypothesis on $I$ in the last theorem is redundant. We want to mention in this context that we establish several conditions each necessary and sufficient for the coincidence of these two topologies on $C(X)$. A typical such condition says that $C_{u^I}(X)=C_{m^I}(X)$ if and only if $C_{m^I}(X)$ is first countable if and only if $X$ is $I$-$pseudocompact$. With the choice $I=C(X)$, this reads, the $m$-topology and $u$-topology on $C(X)$ coincide when and only when $C(X)$ with the $m$-topology is first countable when and only when $X$ is pseudocompact. This fact is also a special case of Theorem 2.2 in \cite{MKJ} with the choice $Y=\mathbb{R}$.
	 
	In Section 2 of this article we recall the definition of a number of topological properties many of which are weaker than the first countablity of a topological space and few other properties akin to the completeness of spaces. We recall also in this section a few standard cardinal functions. 
	 
	   Now let us briefly narrate the organization of the technical part of the present article.
	   
	    In Section 3 of this article, we gather a few seemingly interesting relevant properties of $I$-$pseudocompact$ spaces. A pseudocompact space $X$ is $I$-$pseudocompact$ regardless of the choice of the ideal $I$ in $C(X)$. We show that, there exist $I$-$pseudocompact$ spaces for appropriately chosen ideal $I$ of $C(X)$ which are not pseudocompact. Furthermore, we establish that given a locally pseudocompact non pseudocompact space $X$ and an ideal $I$ in $C(X)$ such that $X$ is $I$-$pseudocompact$, there exists an ideal $J$ in $C(X)$ such that $I\subsetneq J$ and $X$ is $J$-$pseudocompact$. Thus in this case there is no largest ideal $I$ in $C(X)$ such that $X$ is $I$-$pseudocompact$. Additionally we offer characterization of $I$-$pseudocompactness$ of $X$, in terms of two chosen subrings of $C(X)$.
	    
	     In Section 4 of this article, we first show that a metric $d$ is defined on $C(X)$, so that the induced metric topology is identical to the $u^I$-topology on $C(X)$. Furthermore, we check that $C_{m^I}(X)$ is metrizable if and only if $C_{m^I}(X)$ is first countable if and only if $C_{m^I}(X)=C_{u^I}(X)$ if and only if $C_{m^I}(X)$ is a countably tight space. We establish the equivalence of several other conditions each related in some way or other to a few fascinating topological properties of the $m^I$-topology on $C(X)$ and equivalent to the coincidence of the $u^I$-topology and $m^I$-topology on $C(X)$.

	 In Section 5 of this article, we endeavor to interrelate a few well-known cardinal functions on the space $C_{m^I}(X)$. Specifically, we first realize the character of the space $C_{m^I}(X)$ to be identical to the dominating number of the set $X\setminus \bigcap Z[I]$. A special case of this fact with $I=C(X)$ reads: the character of $C(X)$ with the $m$-topology is identical to the dominating number of $X$- this later fact is essentially Theorem 3.1 in \cite{MKJ}. We conclude this section with the realization that four important cardinal functions viz, the weight, density, cellularity and the Lindel\"{o}f number of this space $C_{m^I}(X)$ are identical, a behavior reminiscent of the metric topology. A special case of the last mentioned fact with $I=C(X)$, is essentially Theorem 3.2 in \cite{MKJ}.

	 In Section 6, we correlate a special type of boundedness property viz. $\aleph_0$-boundedness with the  second countability/ separability/ Lindel\"{o}f property of the space $C_{u^I}(X)$ and $C_{m^I}(X)$. Subsequently we discover that two other types of boundedness viz, M-boundedness and H-boundedness and their strict versions too are seen to be equivalent to the hemicompactness/ $\sigma$-compactness of the space $C_{m^I}(X)$.
	 
	 In Section 7, we investigate a few properties related to connectedness/ path connectedness of the spaces $C_{u^I}(X)$ and $C_{m^I}(X)$. We show that at each point of these two spaces, the component, path component and the quasi-component coincide. We also determine when these two spaces become extremally disconnected. 
	 
	 	 In Section 8 of this article, we impose an additional condition on $I$ viz., that $I$ is closed under uniform limits and realize amongst others that $X$ is $I$-$pseudocompact$ if and only if $C_{m^I}(X)$ is \v{C}ech complete. Thus it turns out that in this article the conditon $I$-$pseudocompactness$ of $X$ plays the role of pseudocompactness of $X$. Thereby placing Theorem 2.2 and a portion of Theorem 2.10 in \cite{MKJ} on a more general setting.

	 \section{A few standard topological properties recalled}
	 
	 A subset $S$ of a space $X$ is said to have countable character if there exists a sequence $\{W_n :n\in \mathbb{N} \}$ of open subsets in $X$ such that $S\subseteq W_n$ for all $n\in \mathbb{N}$ and if $W$ is
	 any open set containing $S$, then $W_n \subseteq W$ for some $n\in \mathbb{N}$. A space $X$ is said to be of countable type (pointwise countable type) if each compact set (point) is contained in a compact set having countable character \cite{MKJ}. \\
	 
	 	 A $\pi$-base for a space $X$ is a family of nonempty open sets in $X$ such that every nonempty open set in $X$ contains a member of this family. A point $x \in X$ is said to
	 have a countable local $\pi$-base, if there exists a countable collection $\mathcal{B}_x$ of nonempty
	 open sets in $X$ such that each neighborhood of $x$ contains some member of $\mathcal{B}_x$. If each point of $X$ has a countable local $\pi$-base, then $X$ is said to have countable
	 $\pi$-character \cite{MKJ}. \\
	 
	 A space $X$ is an $r$-space if each point of $X$ has a sequence $\{V_n : n \in \mathbb{N}\}$ of neighborhoods with the property that if $x_n \in V_n$ for each n, then the set $\{x_n : n \in \mathbb{N}\}$ is
	 contained in a compact subset of $X$. A space $X$ is a $q$-space if for each point $x \in X$, there exists a
	 sequence $\{U_n : n \in \mathbb{N}\}$ of neighborhoods of $x$ such that if $x_n \in U_n$ for each $n$, then
	 $\{x_n : n \in \mathbb{N}\}$ has a cluster point. Another property stronger than being a $q$-space is
	 that of being an $M$-space, which can be characterized as a space that can be mapped
	 onto a metric space by a quasi-perfect map (a continuous closed map in which inverse
	 images of points are countably compact). A space $X$ is called a $p$-space if there exists
	 a sequence $(\mathcal{U}_n)$ of families of open sets in a compactification of $X$ such that each
	 $\mathcal{U}_n$ covers $X$ and
	 $\bigcap\limits_{n\in \mathbb{N}}
	 \bigcup\{U \in \mathcal{U}_n : x \in U\} \subseteq X$ for any $x \in  X$ \cite{MKJ}.\\

	A space $X$ is a bi-sequential space if whenever a filter base $\mathcal{F}$ has
	a cluster point $x$ in $X$, then there is a decreasing sequence $\{A_n\}$ of sets
	in $X$ converging to $x$ and such that $F$ intersects $A_n$, for all n and for all
	$F\in \mathcal{F}$. A Space $X$ is a countably bi-sequential space, whenever $\{F_n\}$ is a
	decreasing sequence of sets in $X$ having $x$ as a common accumulation
	point, there exists a decreasing sequence $\{A_n\}$ which converges to $x$
	and such that $A_n$ intersects $F_n$ for all n. A topological space $X$ is called a Fr\'echet  space if for each $A\subseteq X$, and $x \in \cl_X A$
	 there is a sequence $(a_n)$ in $A$ converging to $x$; $X$ is called a sequential space if $A \subseteq X$
	 is closed provided $A$ contains the limits of all convergent sequences from A; $X$ is
	 called a $k$-space if $A \subseteq X$ is closed provided $A \cap K$ is closed in $K$ for each compact
	 subset $K$ of $X$; $X$ is said to be countably tight if for each $A \subseteq X$, and $x \in \cl_X A$ there is
	 a countable subset $B$ of $A$ with $x \in \cl_X
	  B$. A space $X$ is a $k '$-space if a point $x$ is an accumulation
	 point of a set $A$ in $X$ then $x$ is an accumulation point of $A\cap K$ for some
	 compact set $K$. A space $X$ is strongly $k '$ if $\{A_n\}$ is a decreasing sequence
	 of sets each having $x$ as a common accumulation point, then there exists
	 a compact set $K$ such that $x \in \cl (K\cap A_n)$ for all $n$ \cite{RMJ}. \\

A space $X$ is called a radial space if whenever $A \subseteq X$ and $x \in \cl_X A$, then there is an ordinal $\kappa$ and a $\kappa$-sequence $(x_\sigma )_{\sigma <\kappa}$ in $A$ such that $x$ is a limit of the sequence. A space is called pseudoradial if whenever $A \subseteq X$ is not closed then there is an ordinal $\kappa$ and a convergent $\kappa$-sequence $(x_\sigma )_{\sigma <\kappa}$ in $A$ whose limit does not belong to $A$ \cite{MKJ}.\\

A space $X$ is called \v{C}ech-complete if $X$ is a $G_\delta$-set in $\beta X$, where $\beta X$ is the Stone-\v{C}ech compactification of $X$. A space $X$ is called locally \v{C}ech-complete if every point
$x\in X$ has a \v{C}ech-complete neighborhood. A topological space X is called hereditarily Baire if every closed subspace of $X$ is a Baire space.\\
  
	 Thus we have the following categorical diagram (see  \cite{MKJ}, \cite{SH}, \cite{RMJ}):

\begin{figure}[H]
	\begin{center}
		\begin{tikzpicture}[node distance=2cm]
			\node (11) {Pointwise Countable Type};
			\node (12) [below of=11, yshift=0.5cm] {Countable Type};
			\node (13) [below of=12, yshift=0.5cm] {Metric Space};
			\node (14) [below of=13, yshift=0.5cm] {First Countable Space};
			\node (15) [below of=14, yshift=0.5cm] {Bi-sequential Space};
			\node (16) [below of=15, yshift=0.5cm] {Countable Bi-sequential Space};
			\node (17) [below of=16, yshift=0.5cm] {Fr\'echet Space};
			\node (18) [below of=17, yshift=0.5cm] {Sequential Space};
			\node (28) [right of=18, xshift=2.5cm] {Countably Tight};
			\node (24) [right of=14, xshift=2.5cm] {Complete metric space};
		
			\node (23) [above of=24, yshift=-0.5cm] {\v{C}ech-complete};
			\node (22) [above of=23, yshift=-0.5cm] {Locally \v{C}ech-complete};
			\node (25) [below of=24, yshift=-0.5cm] {Countable $\pi$-character};
			\node (21) [above of=22, yshift=-0.4cm] {Baire Space};
			\node (01) [left of=11, xshift=-2.5cm] {$r$-space};
			\node (02) [below of=01, yshift=0.5cm] {$M$-space};
			\node (03) [below of=02, yshift=0.5cm] {$p$-space};
			\node (04) [below of=03, yshift=0.5cm] {Radial Space};
			\node (05) [below of=04, yshift=0.5cm] {Pseudoradial Space};
			\node (06) [below of=05, yshift=0.5cm] {Strongly $k'$-space};
			\node (07) [below of=06, yshift=0.5cm] {$k'$-space};
			\node (08) [below of=07, yshift=0.5cm] {$k$-space};
			\node (00) [left of=02, xshift=-0.3cm] {$q$-space};
			\draw [arrow] (13) -- (12);
			\draw [arrow] (12) -- (11);
			\draw [arrow] (13) -- (14);
			\draw [arrow] (14) -- (15);
			\draw [arrow] (15) -- (16);
			\draw [arrow] (14) -- (25);
			\draw [arrow] (16) -- (17);
			\draw [arrow] (17) -- (18);
			\draw [arrow] (18) -- (28);

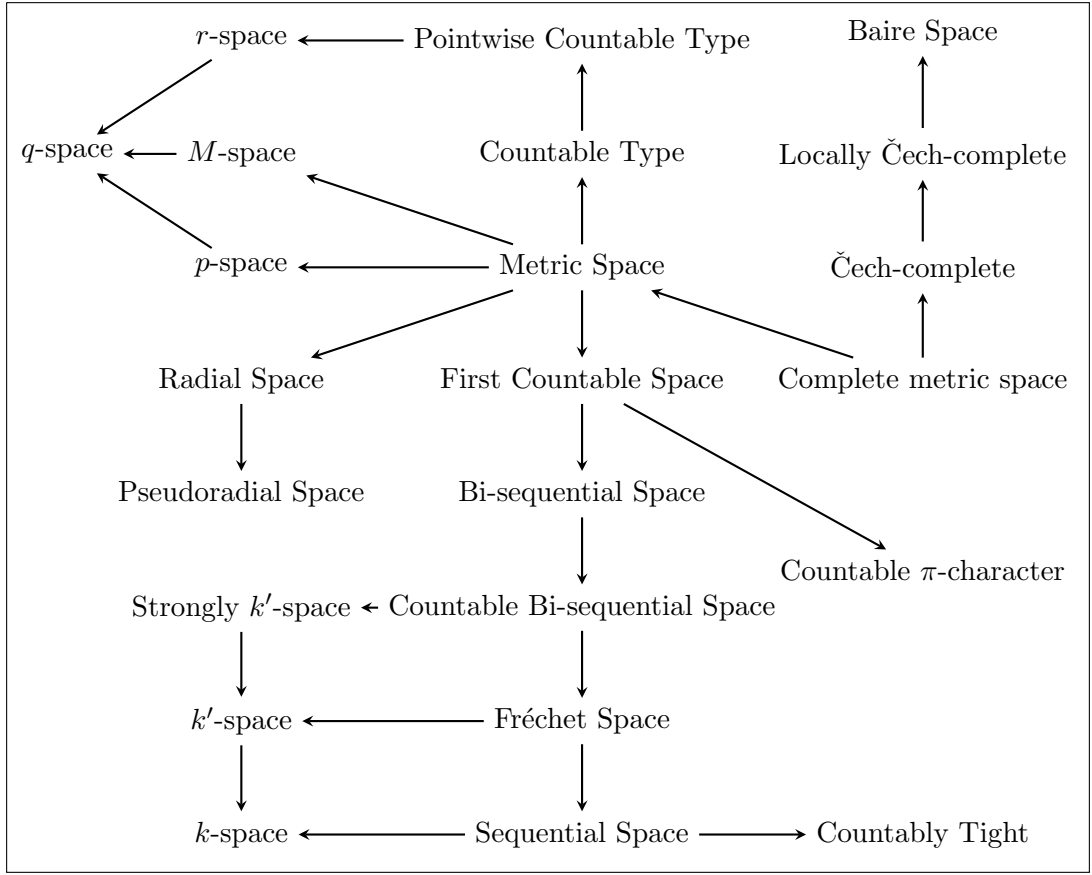
			\draw [arrow] (23) -- (22);
			\draw [arrow] (22) -- (21);
			\draw [arrow] (24) -- (13);
			\draw [arrow] (13) -- (02);
			\draw [arrow] (13) -- (03);
			\draw [arrow] (13) -- (04);
			\draw [arrow] (04) -- (05);
			\draw [arrow] (11) -- (01);
			\draw [arrow] (01) -- (00);
			\draw [arrow] (02) -- (00);
			\draw [arrow] (03) -- (00);
			\draw [arrow] (16) -- (06);
			\draw [arrow] (06) -- (07);
			\draw [arrow] (07) -- (08);
			\draw [arrow] (17) -- (07);
			\draw [arrow] (24) -- (23);
			\draw [arrow] (18) -- (08);
			\draw (-7.6,0.5)--(6.7,0.5)--(6.7,-11)--(-7.6,-11)--(-7.6,0.5);
		\end{tikzpicture}
	\end{center}
\end{figure}
 Now let us recall the definitions of some well-known cardinal functions.\\ 
 
 For any space \( X \) and any point \( x \) in \( X \), the \textit{character} of \( x \) in \( X \), denoted by 
$
 \chi(X,x),
 $
 is defined by
 \[
 \chi(X,x) = \aleph_0 + \min\{|\mathscr{B}_x| : \mathscr{B}_x \text{ is a base for } X \text{ at } x\}.
 \]
 
 The \textit{character} \( \chi(X) \) of \( X \) is defined by
 $
 \chi(X) = \sup\{\chi(X,x) : x \in X\}
 $. Clearly a space \( X \) is \textit{first countable} if \( \chi(X) = \aleph_0 \).
 
 The \textit{weight} of a space \( X \) is defined by
 \[
 w(X) = \aleph_0 + \min\{|\mathscr{B}| : \mathscr{B} \text{ is a base for } X\}.
 \]
 
 A space \( X \) is \textit{second countable} if and only if \( w(X) = \aleph_0 \).
 
 The \textit{density} \( d(X) \) of a space \( X \) is defined by
 \[
 d(X) = \aleph_0 + \min\{|D| : D \text{ is a dense subset of } X\}.
 \]
 
 A space \( X \) is \textit{separable} if and only if \( d(X) = \aleph_0 \).
 
 The \textit{Lindelöf number} \( L(X) \) of \( X \) is defined by
 \[
 L(X) = \aleph_0 + \min\{ \mathfrak{m} : \text{every open cover of } X \text{ has a subcover of cardinality } \leq \mathfrak{m} \}.
 \]
 
 A space is \textit{Lindelöf} if and only if \( L(X) = \aleph_0 \).

 For a space $X$, the cellularity of $X$, denoted by $c(X)$, is defined by $$c(X) =
 \aleph_0 + \sup\{|\mathcal{U} | : \mathcal{U} \text{ is a family of pairwise disjoint nonempty open subsets of } X\}.$$ A
 space $X$ has countable chain condition if and only if $c(X) =\aleph_0$. 
 
  We reproduce the following definitions related to various types of boundedness from \cite{KH}. A topological group $(X,\cdot,\tau)$ is said to be:
 \begin{enumerate}
 	\item  $\aleph_0$-bounded or $\omega$-$narrow$ if for each neighborhood $V$ of the identity element $e\in X$ there is a countable set $A \subseteq X$ such that $X = A \cdot V$.
 	\item $Menger$ $bounded$ or $M$-$bounded$ if if for each sequence $\{V_n:n\in \mathbb{N}\}$ of
 	neighborhoods of the identity element $e \in X$ there is a sequence $\{A_n:n\in \mathbb{N}\}$ of finite subsets of $X$ such that $X=\bigcup\limits_{n\in \mathbb{N}} A_n\cdot V_n$.
 	\item $ Hurewicz$ $bounded$ or $H$-$bounded$ if for each sequence $\{V_n:n\in \mathbb{N}\}$ of
 	neighborhoods of the identity element $e \in X$ there is a sequence $\{A_n:n\in \mathbb{N}\}$ of finite subsets of
 	$X$ such that each $x \in X$ belongs to all but finitely many sets $A_n\cdot V_n$.
 	\item $Rothberger$ $bounded$ or $R$-$bounded$ if for each sequence $\{V_n:n\in \mathbb{N}\}$
 	of neighborhoods of the identity element $e \in X$ there is a sequence $\{x_n:n\in \mathbb{N}\}$ of elements of $X$
 	such that $X =\bigcup\limits_{n\in \mathbb{N}}	x_n \cdot V_n$.
 	
 \end{enumerate}
 
 To each of the above properties one can correspond a game on $X$. Let us
 give an idea of this for $M$-$boundedness$. Two players, A and B, play a round
 for each $n\in \mathbb{N}$. In the $n$-th round player A chooses a basic neighborhood $V_n$ of
 identity element $e\in X$, and B responds by choosing a finite set $A_n$ in $X$. B wins a play
 $$V_1, A_1;V_2, A_2; \cdots ;V_n, A_n; \cdots$$
 if $X =\bigcup\limits_{n\in \mathbb{N}} A_n · U_n$; otherwise A wins.

 A topological group $(X, \cdot, \tau )$ is said to be $strictly$ $M$-$bounded$ if B has a
 winning strategy in the above game.

 \section{I-pseudocompact spaces}
 
 We recall that for a given ideal $I$ in $C(X)$, we call the space $X$ $I$-$pseudocompact$ when $X\setminus \bigcap Z[I] $ is a bounded subset of $X$. It is plain that a pseudocompact space $X$ is $I$-$pseudocompact$ with $I=C(X)$. In this section we intend to construct various non pseudocompact spaces $X$, which are $I$-$pseudocompact$ for appropriate choices of the proper ideal $I$ in $C(X)$. It is trivial that if $X$ is $I$-$pseudocompact$, then it is also $J$-$pseudocompact$ for any ideal $J$ of $C(X)$ contained in $I$. But it seems to be fascinating problem that, if $X$ is $I$-$pseudocompact$, then does there exist an ideal $J$ in $C(X)$ with $I \subsetneqq J$ for which $X$ is also $J$-$pseudocompact$. We have given an affirmative answer to this question for a large class of spaces $X$. The following result gives possible candidates for $I$ to render $X$, an $I$-$pseudocompact$ space.
 \begin{theorem}\label{6.1}
 	Suppose $I$ is an ideal in $C(X)$ for which $X$ is $I$-$pseudocompact$. Then $I\subseteq C_\psi(X)\equiv$ the ideal of all functions in $C(X)$ with pseudocompact support.
 \end{theorem}
 \begin{proof}
 	Let $f\in I$. Then $X\setminus Z(f)\subseteq X\setminus \bigcap Z[I]$. Since $X $ is $I$-$pseudocompact$, it follows that $X\setminus Z(f)$ is a bounded subset of $X$ and hence $\cl_X(X\setminus Z(f))\equiv$ support of f is a bounded subset of $X$. It is a standard result proved by Mandelker in 1971 \cite{m} that the support of a continuous function over $X$ is a bounded subset of $X$ if and only if it is a pseudocompact subset of $X$. It follows that $f\in C_\psi(X)$.
 \end{proof}
 \begin{remark}
 	It is easy to construct an example of an ideal $I $ contained in $C_\psi(X)$ for a suitable space $X$ for which $X$ is not $I$-$pseudocompact$. Indeed if we take $X=\mathbb{R}$, then $C_\psi(X)=C_k(X)=$ the ideal of all functions in $C(X)$ with compact support. Since $X=\mathbb{R}$ is locally compact, $I=C_k(X)$ is a free ideal in $C(X)$ \cite[Exercise 4D3]{GJ} which means that $\bigcap Z[I]=\emptyset$, clearly then $X\setminus \bigcap Z[I]=X=\mathbb{R}$ is never a bounded subset of $X$ and therefore $\mathbb{R}$ is not $I$-$pseudocompact$.
 \end{remark}
 
 The next result determines completely the class of spaces $X$, for which $X$ becomes $I$-$pseudocompact$ for some non zero ideal $I$ in $C(X)$.
 \begin{theorem}\label{6.3}
 	The following statements are equivalent for a space $X$.
 	\begin{enumerate}
 		\item There exists a non zero ideal $I$ in $C(X) $ such $X$ is $I$-$pseudocompact$.
 		\item There exists at least one $x\in X$ such that $X$ is locally pseudocompact at $x$.
 	\end{enumerate}
 \end{theorem}
 \begin{proof}
 	$(1)\Rightarrow (2)$: Let $X$ be $I$-$pseudocompact$ for some ideal $I\neq\{0\}$ in $C(X)$. Now by Theorem \ref*{6.1}, $I\subseteq C_\psi(X)$. Choose $f\in I$ for some $f\neq0$. Then $X\setminus Z(f)\neq \emptyset$. Choosing any point $x\in X\setminus Z(f)$, we see that $\cl_X(X\setminus Z(f))\equiv$ support of $f$ is a pseudocompact neighborhood of $x$, in other words $X$ is locally pseudocompact at $x$.
 	
 	$(2)\Rightarrow (1)$: Suppose $X$ is locally pseudocompact at some $x_0\in X$. Since $X$ is Tychonoff, there exists $f\in C(X)$ such that $x_0\in X\setminus Z(f)$ and $\cl_X(X\setminus Z(f))$ is pseudocompact, consequently $X\setminus Z(f)$ is a bounded subset of $X$. Let $I=\langle f \rangle\equiv$ the principal ideal generated by $f$ in $C(X)$, then $I\neq\{\underline{0}\}$ and $\bigcap Z[I]=Z(f)\implies X\setminus \bigcap Z[I]=X\setminus Z(f)$, a bounded subset of $X$. Thus $X$ becomes $I$-$pseudocompact$. 
 \end{proof}
  An ideal $I$ in $C(X)$ is called fixed if  $\bigcap Z[I]\equiv\bigcap \limits_{g\in I} Z(g)\neq \emptyset$, otherwise $I$ is called a free ideal in $C(X)$. It is trivial that if $I$ is a free ideal in $C(X)$ and $X$ is $I$-$pseudocompact$, then $X$ is pseudocompact. In fact more generally if $I$ is an essential ideal in $C(X)$ (this means that $I$ cuts every nonzero ideal in $C(X)$ non trivially) and $X$ is $I$-$pseudocompact$, then $X$ is pseudocompact. [An ideal $I$ in $C(X)$ is an essential ideal in $C(X)$ if and only if $\bigcap Z[I]$ is no where dense in $X$ \cite{AA}. Each free ideal in $C(X)$ is an essential ideal]. Therefore for a non pseudocompact space $X$, if $X$ is $I$-$pseudocompact$ for a non zero ideal $I$ in $C(X)$, then $I$ must be a fixed ideal in $C(X)$. If in addition $X$ is locally pseudocompact, then $C_\psi(X)$ is a free ideal in $C(X)$ \cite[Remark 4.3]{SB} and therefore any ideal $I\neq\{\underline{0}\}$ in $C(X)$ which makes $X$, $I$-$pseudocompact$ is properly contained in $C_\psi(X)$.
 
 The following fact shows that for such a space $X$, there is no largest ideal $I$ (contained in $C_\psi(X)$) for which $X$ is $I$-$pseudocompact$.
 
 \begin{theorem}
 	Let $X $ be a locally pseudocompact non pseudocompact space. Then given any ideal $I$ in $C(X)$ such that $X$ is $I$-$pseudocompact$, there exists an ideal $J$ in $C(X)$ such that $I\subsetneqq J$ and $X$ is $J$ pseudocompact and therefore there exists an increasing sequence of ideals $I_1\subsetneqq I_2\subsetneqq I_3\cdots$ of $C(X)$ such that $X$ becomes $I_n$-pseudocompact for each $n\in \mathbb{N}$.
 \end{theorem}
 \begin{proof}
 	Suppose $I\neq\{\underline{0}\}$ is an ideal in $C(X)$ such that $X$ is $I$-$pseudocompact$. The existence of such ideal is ensured from Theorem \ref*{6.3}. Then from the comments made preceding the statement of the present theorem, we can say that $I\subsetneq C_\psi(X)$. Choose $f\in C_\psi(X)\setminus I$ and set $J=\langle I,f\rangle $$\equiv$ the smallest ideal in $C(X)$ containing $I$ and the function $f$. Then $\bigcap Z[J]=\bigcap Z[I]\cap Z(f)$. This implies that $X\setminus \bigcap Z[J]=(X\setminus \bigcap Z[I])\cup (X\setminus Z(f))	$. Since $X$ is $I$-$pseudocompact$, $X\setminus \bigcap Z[I]$ is a bounded subset of $X$, on the other hand $f\in C_\psi(X)$ implies $\cl_X(X\setminus Z(f ))$ is a pseudocompact subset of $X$, equivalently a bounded subset of $X$. Altogether, $X\setminus \bigcap Z[J]$ becomes a bounded subset of $X$. Hence $X$ is $J$-$pseudocompact$.
 \end{proof}
  We shall conclude this section after furnishing two characterization of $I$-$pseudocompact$ space $X$, in terms of two chosen subrings of $C(X)$. 
  
  For the given ideal $I$ in $C(X)$, let $$C_{\psi^I}(X)=\{f\in C(X): \overline{(X\setminus Z(f))\cap (X\setminus\bigcap Z[I])  }\text{ is a bounded subset of } X \}$$ It is easy to see that with $I=C(X)$, $C_{\psi^I}(X)$ reduces to the ideal $C_\psi(X)$ of $C(X)$. It needs a routine computation to check that $C_{\psi^I}(X)$ is an ideal (may be improper) in $C(X)$ and $C_\psi(X)\subseteq C_{\psi^I}(X)$.
  \begin{theorem}\label{Ipsi}
  For an ideal $I$ in $C(X)$, $X$ is $I$-$pseudocompact$ if and only if $C_{\psi^I}(X)=C(X)$. 
  \end{theorem}
\begin{proof}
	Let $X$ be $I$-$pseudocompact$, then $X\setminus \bigcap Z[I]$ is a bounded subset of $X$ and hence, $\overline{X\setminus\bigcap Z[I]}$ is also bounded subset of $X$. It follows that for any $f\in C(X)$, $\overline{(X\setminus Z(f))\cap (X\setminus\bigcap Z[I])}$ is a bounded subset of $X$ and hence $C(X)=C_{\psi^I}(X)$. Conversely let, $C(X)=C_{\psi^I}(X)$, then  for the constant function $\underline{1}$, $\overline{(X\setminus Z(\underline{1}))\cap (X\setminus\bigcap Z[I])}$ is a bounded subset of $X$, meaning that ${X\setminus\bigcap Z[I]}$ is bounded subset of $X$, hence $X$ is $I$-$pseudocompact$. 
\end{proof}
The following alternative description of $C_{\psi^I}(X)$ will be helpful to us towards further investigation on this ideal in $C(
X)$.
\begin{theorem}\label{Eqcpsi}
	$C_{\psi^I}(X)=\{f\in C(X): \text{ for each }g\in C(X), fg\text{ is bounded on }X\setminus\bigcap Z[I]\}$. (with the choice of $I=C(X)$, this reduces to $C_\psi(X)=\{f\in C(X):\text{ for each }g\in C(X)\text{ }fg\text{ is}\newline\text{ bounded on }X\}$, a standard result obtained in \cite{JM})
\end{theorem}
\begin{proof}
	If $f\in C(X)$ is such that for each $g\in C(X)$, $fg$ is bounded on $X\setminus \bigcap Z[I]$, which surely implies that $f$ is bounded on $\overline{(X\setminus Z(f))\cap (X\setminus\bigcap Z[I])}$ and therefore, $f\in C_{\psi^I}(X)$. Conversely, if $f\in C_{\psi^I}(X)$, then $\overline{(X\setminus Z(f))\cap (X\setminus\bigcap Z[I])}$ is a bounded subset of $X$. This implies for each $g\in C(X)$ that, $fg$ bounded on $\overline{(X\setminus Z(f))\cap (X\setminus\bigcap Z[I])}$. Since $fg$ is zero in $Z(f)$, this further implies that $fg$ is bounded in $X\setminus\bigcap Z[I]$.
\end{proof}
It is standard result in the theory of rings of continuous function that $C_\psi(X)$ is the largest ideal in $C(X)$ contained in $C^*(X)$ \cite[Theorem 3.3.7]{pn}. The next result places this fact on a wider setting. 
 Let $$C_I^*(X)=\{f\in C(X): f\text{ is bounded on }X\setminus\bigcap Z[I]\}$$ It is clear that $C^*(X)\subseteq C^*_I(X)$ and $C^*_I(X)$ is a subring of $C(X)$ and $C_{\psi^I}(X)\subseteq C_I^*(X)$- this last implication follows from Theorem \ref{Eqcpsi}.
 \begin{theorem}
 $C_{\psi^I}(X)$ is the largest ideal in $C(X)$ contained in $C_I^*(X)$.
 
 \end{theorem}
 \begin{proof}
 	It is already noted that $C_{\psi^I}(X)$ is an ideal in $C(X)$. Let $J$ be an ideal in $C(X)$ contained in $C_I^*(X)$. We need to check that $J\subseteq C_{\psi^I}(X)$. Let $f\in J$. Then for each $g\in C(X)$, $fg\in J$, which implies that $fg\in C_I^*(X)$ and therefore, $fg$ is bounded on $X\setminus\bigcap Z[I]$. This implies in view of Theorem \ref{Eqcpsi} that $f\in C_{\psi^I}(X)$.
 \end{proof}
The following theorem offers several equivalent descriptions of $I$-$pseudocompact$ spaces.
\begin{theorem}
	The following four statements are equivalent for an ideal $I$ in $C(X)$-\begin{enumerate}
		\item $X$ is $I$-$pseudocompact$.
		\item $C^*_I(X)=C(X)$.
		\item $C_{\psi ^I}(X)=C^*_I(X)$.
		\item $C_I^*(X)$ is an ideal in $C(X)$.
	\end{enumerate}
\end{theorem}
\begin{proof}
	$(1)\Rightarrow (2)$ follows from Theorem \ref{Ipsi} and the fact that $C_{\psi^I}(X)\subseteq C^*_I(X)$. If $(2)$ holds, then each $f\in C(X)$ is bounded on $X\setminus\bigcap Z[I]$ and this means that $X$ is $I$-$pseudocompact$. Thus $(1)\iff (2)$. $(1)\Rightarrow(3)$ follows from Theorem \ref{Ipsi} and the equivalence of $(1)\iff (2)$. $(3)\Rightarrow(4)$ is trivial because $C_{\psi^I}(X)$ is already an ideal in $C(X)$. $(4)\Rightarrow(1)$: Let $(4)$ be true and $f\in C(X)$. Then since $1\in C_I^*(X)$, it follows that $f\in C_I^*(X)$ and hence $f$ is bounded on $X\setminus \bigcap Z[I]$. Since $f\in C(X)$ is arbitrary, it follows that $X\setminus\bigcap Z[I]$ is a bounded subset of $X$ i.e., $X$ is $I$-$pseudocompact$.
\end{proof}
 \section{ Metrizability of $C_{u^I}(X)$ and $C_{m^I}(X)$ }
 
 \begin{theorem}\label{m}
 	 	The space $C_{u^I}(X)$ is metrizable, indeed the function $d:C(X)\times C(X)\rightarrow[0,\infty)$ given by the formula $$d(f,g)=\begin{cases}
 		1,\text{ when }f-g\notin I\\
 	\min\{	\sup\{|f(x)-g(x)|:x\in X\setminus \bigcap Z[I]\}, 1\}, \text{ when } f-g\in I
 	\end{cases}$$ is a metric on $C(X)$, which defines the $u^I$-topology on $C(X)$.
 \end{theorem}

The proof of this theorem is routine and therefore omitted.\\

The following simple lemma will be helpful towards proving the equivalence of a number of statements in this section.
\begin{lemma}\label{Lemma}
	Let $I$ be an ideal in $C(X)$ and $c\in X\setminus \bigcap Z[I]$. Then there exists $f\in I$ such that $|f|\leq1$ on $X$ and $f(c)=\frac{1}{2}$.
	
\end{lemma}
\begin{proof}
	$c\in X\setminus \bigcap Z[I]$ implies that there exists a $g\in I$ such that $g(c)\neq 0$. Let $k=\frac{1}{g(c)}g$, then $k\in I$. Suppose $f=\frac{k}{1+|k|}$, then $f\in I$, $|f|\leq1$ on $X$ and $f(c)=\frac{1}{2}$.
\end{proof}
\begin{theorem}\label{Ipsu}
	Given an ideal $I$ in $C(X)$, the following statements are equivalent:
	\begin{enumerate}
		\item $X$ is $I$-pseudocompact.
		\item $C_{u^I}(X)=C_{m^I}(X)$.
		\item $C_{m^I}(X)$ is first countable.
		\item $C_{u^I}(X)$ is a topological ring.
		\item $C_{m^I}(X)$ is a bi-sequential space.
		\item $C_{m^I}(X)$ is a countably bi-sequential space.
		\item $C_{m^I}(X)$ is a Fr\'{e}chet space.
		\item $C_{m^I}(X)$ is a sequential space.
		\item $C_{m^I}(X)$ is strongly-$k'$.
		\item $C_{m^I}(X)$ is a $k'$-space.
		\item $C_{m^I}(X)$ is a k-space.
		\item $C_{m^I}(X)$ is countably tight.
		\item $C_{m^I}(X)$ is a metrizable space.
		\item $C_{m^I}(X)$ is an M-space.
		\item $C_{m^I}(X)$ is a p-space.
		\item $C_{m^I}(X)$ is an r-space.
		\item $C_{m^I}(X)$ is a q-space.
		\item $C_{m^I}(X)$ is of countable type.
		\item $C_{m^I}(X)$ is of pointwise countable type.
		\item $C_{m^I}(X)$ has a dense subset of pointwise countable type.
		\item $C_{m^I}(X)$ has a countable $\pi$-character.
		\item $C_{m^I}(X)$ has a dense subset which is of countable $\pi$-character.
		\item $C_{m^I}(X)$ is a radial space.
		\item $C_{m^I}(X)$ is a pseudoradial space.
		
	\end{enumerate}
	
\end{theorem}
\begin{proof}
	$(3)\Rightarrow (5)\Rightarrow(6)\Rightarrow(7)\Rightarrow(8)\Rightarrow(12)$,
	
	$(3)\Rightarrow(5)\Rightarrow(6)\Rightarrow(9)\Rightarrow(10)\Rightarrow(11)$,
	
	$(13)\Rightarrow (14)\Rightarrow(17)$,
	
	$(13)\Rightarrow(15)\Rightarrow(17)$,
	
	$(13)\Rightarrow(18)\Rightarrow(19)\Rightarrow(16)\Rightarrow(17)$,
	$(13)\Rightarrow(23)\Rightarrow(24)$,
	
	All the above implications follow from the diagram of arrows drawn in Section 2.
	
	$(1)\Rightarrow(2)$: Established in Theorem 3.7 in \cite{PA2023}.
	
	$(2)\Rightarrow(3) $ and $(2)\Rightarrow (13)$ follow immediately because of Theorem \ref*{m}.
	
	$(2)\Rightarrow (4)$ follows because $C_{m^I}(X)$ is a topological ring.
	
	$(4)\Rightarrow (1)$: Let $f\in C(X)$. We shall show that $f$ is bounded on $X\setminus \bigcap Z[I]$. The hypothesis (4) implies that there exist $\epsilon_1,\epsilon_2>0$ in $\mathbb{R}$ such that $B_u(0,I,\epsilon_1)\cdot B_u(f,I,\epsilon_2)\subseteq B_u(0,I,1)$. Choose $c\in X\setminus \bigcap Z[I] $ arbitrarily. Then by Lemma \ref{Lemma}, there exists a $g\in I$ such that $|g|\leq1$ on $X$ and $g(c)=\frac{1}{2}$. Therefore $\frac{1}{2}\epsilon_1\cdot g\in B_u(0,I,\epsilon_1)$ and it is trivial that $f\in B_u(f,I,\epsilon_2)$. This yields that $\frac{1}{2}\epsilon_1 \cdot g\cdot f\in B_u(0,I,1)$ and hence $|\frac{1}{2}\epsilon_1 g(c)f(c)|<1 $, consequently $|f(c)|<\frac{4}{\epsilon_1}$. Thus $|f(x)|<\frac{4}{\epsilon_1} $ on $X \setminus \bigcap Z[I]$ and therefore $X$ is $I$-$pseudocompact$.
	
	$(11)\Rightarrow(12)$: Since $C_{m^I}(X)$ is a submetrizable space because of Theorem \ref{m}, it follows that each point of $C_{m^I}(X)$ is a $G_\delta$-set. Furthermore, $C_{m^I}(X)$ is a $T_3$ space because every $T_0$ topological group is regular. It is proved in \cite[Theorem 2.1 (l)$\Rightarrow$(m)]{MKJ} that a regular $k$-space in which points are $G_\delta$ is countably tight.
	
	$(12)\Rightarrow(1)$: Suppose $(1)$ is false. Then by Corollary 1.20 in \cite{GJ}, $X\setminus \bigcap Z[I] $ contains a copy of $\mathbb{N}$: $D\equiv\{x_1,x_2,\cdots\}$, C-embedded in $X$. By Lemma \ref{Lemma} for each $n\in \mathbb{N}$, there exists $f_n\in I$ such that $|f_n|\leq 1$ on $X$ and $f_n(x_n)=\frac{1}{2}$. Let $L=\{g\in C(X):g(x_n)\neq 0 \text{ for each }n\in \mathbb{N}\}$ and $G=\{|f_n\cdot g|:n\in \mathbb{N}, g\in L\}.$ Now for any $u\in C_+(X)$, $\frac{1}{2}|f_1\cdot u|\in B_m(0,I,u)\cap G$ and hence $0\in \bar{G}\equiv \cl_{C_{m^I}(X)}G$. We claim that $0$ does not belong to the closure of any countable subset of $G$ in $C_{m^I}(X)$ and hence $C_{m^I}(X)$ is not countably tight. Towards proving the claim let $G_0=\{|f_{n_1}\cdot g_1|,|f_{n_2}\cdot g_2|, \cdots\}$ be any countable subset of $G$, here $g_n\in L$ for each $n\in \mathbb{N}$. Let $\delta:D\rightarrow \mathbb{R} $ be defined as follows: $$\delta(x_{n_i})=\begin{cases}
		\frac{1}{4}|f_{n_i}(x_{n_i})||g_i(x_{n_i})|, \text{ when }i\in \mathbb{N}\\
		1,\text{ otherwise }
	\end{cases}$$ Then $\delta$ has an extension to a $v\in C_+(X)$. We claim that $B_m(0,I,v)\cap G_0=\emptyset$ and hence $0\notin cl_{C_{m^I}(X)}G_0$. Towards proving, we argue by contradiction and assume that some $f_{n_k}\cdot g_k\in B_m(0,I,v)$. Then $|(f_{n_k}\cdot g_k)(x_{n_k})|<v(x_{n_k})$ and hence $\frac{1}{2}|g_k(x_{n_k})|<\frac{1}{4}\cdot\frac{1}{2}|g_k(x_{n_k})|$, a contradiction.

To complete the present theorem we need to establish a few other implication relations, which could be accomplished by closely following the arguments made to prove the corresponding implication relations in \cite[Theorem 2.1]{MKJ} and making a few obvious modifications.
\end{proof}

\section{A few cardinal functions}
\begin{definition}
	A subset $F$ of $C(X)$ is called dominating on a subset $A$ of $X$ if given $f\in C(X)$, there exists $g\in F$ such that $f(x)\leq g(x)$ for each $x\in A$. We set $$dn(A)=\aleph_0+\min\{|F|:F\subseteq C(X) \text{ and } F \text{ is dominating on the set }A\}$$ and call it the dominating number of $A$. It is easy to verify that for any subset $A$ of $X$, we have $dn(A)\leq dn(X)$. The following result is a straight forward generalization of Proposition 3.1 in \cite{MKJ}, whose proof also goes through in our case.
\end{definition}
\begin{theorem}
	$X$ is $I$-$pseudocompact$ if and only if $dn(X\setminus \bigcap Z[I])=\aleph_0$.
\end{theorem}
\begin{corollary}
	 $C_{m^I}(X)=C_{u^I}(X)$ if and only if $dn(X\setminus \bigcap Z[I])=\aleph_0$.
\end{corollary}

Our next theorem is a generalization of Theorem 3.1 in \cite{MKJ}. Our proof is modelled strongly  by the line of arguments given in the proof of Theorem 3.1 in \cite{MKJ}, nevertheless we provide a sketch of an independent proof of ours, because we need to incorporate certain additional points to suit our present problem.

\begin{theorem}
	$\chi(C_{m^I}(X))=dn(X\setminus \bigcap Z[I])$.
\end{theorem}
\begin{proof}
	By closely following the proof of the first part of Theorem 3.1 in \cite{MKJ}, it is not hard to show $\chi(C_{m^I}(X))\leq dn(X\setminus\bigcap Z[I])$. To prove the reverse inequality, let $\mathcal{B}$ be a local open base about the point $\underline{0}$ in the space $C_{m^I}(X)$ with $|\mathcal{B}|=\chi(C_{m^I}(X))$ and we can assume that each $B\in \mathcal{B}$ is of the form $B_m(\underline{0},I,u_B)$ with $u_B\in C_+(X)$. We assert that the set $\{\frac{4}{u_B}:B\in \mathcal{B}\}$ is a dominating set on $X\setminus \bigcap Z[I]$ and this proves the desired reverse inequality. Now towards the proof of this assertion let $f\in C(X)$ with $f(x)>0$ for each $x\in X$. Then $B_m(\underline{0},I,\frac{1}{f})$ is an open neighborhood of $0$ in $C_{m^I}(X)$, consequently there exists an $u_B\in C_+(X)$ such that $B_m(\underline{0},I,u_B)\subseteq B_m(\underline{0},I,\frac{1}{f})$. Now choose $c\in X\setminus \bigcap Z[I]$ arbitrarily. Then by Lemma \ref{Lemma}, there exists $h\in I$ such that $h(c)=\frac{1}{2}$ and $|h|\leq1$ on $X$. It follows that $\frac{1}{2}hu_B\in B_m(\underline{0},I,u_B)\subseteq B_m(\underline{0},I,\frac{1}{f})$. Consequently $\frac{1}{2}h(c)u_B(c)<\frac{1}{f(c)}$ and hence $f(c)<\frac{4}{u_B(c)}$. This inequality is true for any point $c\in X\setminus \cap Z[I]$ and we are done.
\end{proof}
\begin{corollary}
	$C_{u^I}(X)=C_{m^I}(X)$ if and only if $C_{m^I}(X)$ is first countable.
\end{corollary}

The next three cardinal inequalities are a little bit generalized version of their counterparts with $I=C(X)$ as established in Proposition 3.3 and Proposition 3.4 in \cite{MKJ} and whose proofs also go through in our case.
\begin{theorem}
	$d(C_{m^I}(X))\leq \chi(C_{m^I}(X))\cdot c(C_{m^I}(X)).$
\end{theorem} 
\begin{theorem}
	$d(C_{m^I}(X))\leq \chi(C_{m^I}(X))\cdot L(C_{m^I}(X))$.
\end{theorem}
\begin{theorem}
	$dn(X\setminus \bigcap Z[I])\leq c(C_{m^I}(X))$.
\end{theorem}
The cardinal inequality that we are going to present now is also a generalization of its counterpart viz the second portion of Proposition 3.4 in \cite{MKJ}. But there is a small error in its proof while making a proof of this portion of the proposition, the authors have written that given a dominating subset $F$ of $C(X)$, there exists a subset $G$ of $F$ for which $\{B_m(g,1):g\in G\}$ is an open cover of $C_f(X)$. But this is not correct, because we can choose a dominating subset $F$ of $C(X)$ such that for each $f\in F$, $f\geq 3$. In that case $1\notin B_m(f,1)$ for each $f\in F$. Nevertheless we give a correct proof of the same inequality in its generalized version.
\begin{theorem}
	$dn(X\setminus \bigcap Z[I])\leq L(C_{m^I}(X))$.
\end{theorem}
\begin{proof}
	 Given $\{B_m(g,I,1):g\in C(X)\}$ is an open covering of $C_{m^I}(X)$, there is a subfamily $G$ of $C(X)$ such that $|G|\leq L(C_{m^I}(X))$ and $\{B_m(g,I,1):g\in G\}$ is a cover of $C(X)$. It is not hard to check that $G$ is a dominating set on $X\setminus \bigcap Z[I] $ because for any $h\in C(X)$, $h+1\in B_m(h_0,I,1)$ for some $h_0\in G$. Therefore $h(x)+1-h_0(x)<1$ for all $x\in X\setminus \bigcap Z[I]$ and hence $h(x)<h_0(x)$ for each $x\in X\setminus \bigcap Z[I]$.
\end{proof}
 A straightway generalization of Proposition 3.2 in \cite{MKJ} yields the following cardinal equality.
 \begin{theorem}
 	$w(C_{m^I}(X))=\chi(C_{m^I}(X))\cdot d(C_{m^I}(X))$.
 \end{theorem}
Combining all the above cardinal inequalities, the following result is obtained.
\begin{theorem} \label{equi}
	For any space $X$ and for an ideal $I$ in $C(X)$,
	$$w(C_{m^I}(X))=d(C_{m^I}(X))=c(C_{m^I}(X))=L(C_{m^I}(X)).$$
\end{theorem}

\section{The question of boundedness and compactness for subsets of $C_{m^I}(X)$}
Since the topology of $C_{m^I}(X)$ is finer than that of the space $C_{u^I}(X)$, it follows from Theorem \ref*{m} that $C_{m^I}(X)$ is a submetrizable space. As a submetrizable pseudocompact space is metrizable \cite[Lemma 4.3]{KM} see also \cite{SH}, the following result is obtained almost immediately just by mimicking the proof of Theorem 5.1 in \cite{MKJ}.
\begin{theorem}\label{Compact}
	The following statements are equivalent for a subset $K$ of $C(X)$:
	\begin{enumerate}
		\item $K$ is compact in $C_{m^I}(X)$.
		\item $K$ is sequentially compact in $C_{m^I}(X)$.
		\item $K$ is countably compact in $C_{m^I}(X)$.
		\item $K$ is pseudocompact in $C_{m^I}(X)$.
	\end{enumerate}
\end{theorem} 

It is proved in \cite{A2017}  that every compact subset of $C_{m^I}(X)$ has an empty interior if and only if $X\setminus \bigcap Z[I]$ is infinite. This fact combined with Theorem \ref{Compact} yields the following result immediately.
\begin{theorem}\label{c}
	The following statements are equivalent for an ideal $I$ in $C(X)$:
	\begin{enumerate}
		\item Each compact subset of $C_{m^I}(X)$ has empty interior.
		\item Each sequentially compact subset of $C_{m^I}(X)$ has empty interior.
		\item Each countably compact subset of $C_{m^I}(X)$ has empty interior.
		\item Each pseudocompact subset of $C_{m^I}(X)$ has empty interior.
		\item $X\setminus \bigcap Z[I]$ is infinite.
	\end{enumerate}
\end{theorem}

We are going to show that, within the class of compact metrizable spaces $X$, $C(X)$ is characterized among all the ideals $I$ in $C(X)$ by the requirement that $C_{m^I}(X)$ is $\aleph_0$-bounded. We need a subsidiary result to prove this fact.
\begin{lemma}\label{bdd}
	Let $I$ be a fixed ideal in $C(X)$. Then the space $C_{u^I}(X)$ is never $\aleph_0$-bounded.
\end{lemma} 
\begin{proof}
	Choose a point $c$ from the set $\bigcap Z[I]\equiv \bigcap \limits_{g\in I} Z(g)$. The set $B_u(\underline{0},I,1)$ is an open neighborhood of $\underline{0}$ in $C_{u^I}(X)$. We claim that there does not exist any countable subset $D$ of $C(X)$ for which we can write $C(X)=B_u(\underline{0},I,1)+D$. Suppose the contrary and there exists a countable set $\{f_1,f_2,\cdots\}$ in $C(X)$ with $C(X)=B_u(\underline{0},I,1)+\{f_1,f_2,\cdots\}$. We select a real number $r\neq 0$ such that $r\neq f_i(c)$ for each $i\in \mathbb{N}$. Then $\underline{r}=g+f_k$ for some $k\in \mathbb{N}$ and $g\in B_u(\underline{0},I,1)$. This implies that $\underline{r}-f_k\in I$ and hence $r=f_k(c)$, a contradiction.
\end{proof}
\begin{theorem}
	The following statements are equivalent for an ideal $I$ in $C(X)$:
	\begin{enumerate}
		\item $C_{m^I}(X)$ is $\aleph_0$-$bounded$.
		\item $C_{u^I}(X)$ is $\aleph_0$-$bounded$.
		\item $X$ is compact, metrizable and $I=C(X)$.
		\item $C_{m^I}(X)$ is separable.
		\item $C_{m^I}(X)$ satisfies countable chain condition.
		\item $C_{m^I}(X)$ is second countable.
		\item $C_{m^I}(X)$ is Lindel\"{o}f.
		\item $C_{u^I}(X)$ is separable.
		\item $C_{u^I}(X)$ is Lindel$\ddot{o}$f.	
			\item $C_{u^I}(X)$ satisfies countable chain condition.
		\item $C_{u^I}(X)$ is second countable.
	\end{enumerate}
\end{theorem}
\begin{proof}
	Since $C_{u^I}(X)$ is weaker than $C_{m^I}(X)$, $(1)\implies (2)$ is trivial.
	
	$(2)\implies (3)$: Let $(2)$ be true. Therefore from Lemma \ref{bdd}, it follows that $I$ is a free ideal in $C(X)$. Now the $u$-topology on $C(X)$ is weaker than the $u^I$-topology implies in view of the assumed condition (2) that $C_u(X)$ is $\aleph_0$-$bounded$. It follows from \cite[Theorem 2.4]{KH} that $X$ is compact and metrizable. The compactness of $X$ further implies that each proper ideal in $C(X)$ is fixed \cite[Theorem 4.8]{GJ}. Hence $I=C(X)$.
	
	$(3)\implies (1)$ follows immediately from Theorem 2.4 in \cite{KH}, because with $I=C(X)$, $C_{m^I}(X)=C_m(X)$.
	
	Equivalence of the statements $(4),(5),(6)$ and $(7)$ follow directly from Theorem \ref{equi}. Since every metric space with countable chain condition is separable, this implies the equivalence of (8) and (10). The statements (8), (9) and (11) are clearly equivalent because $C_{u^I}(X)$ is a metrizable space. $(3)\implies (4)$ follows from Theorem 3.6 in \cite{MKJ}. $(6)\implies (10)$ is immediate because $u^I$-topology is weaker than the $m^I$-topology. Since every separable topological group is $\aleph_0$-$bounded$ \cite[Corollary 3.4.8]{A}, $(8)\implies (2)$ follows.
\end{proof}

For a topological group, $\sigma$-compactness $\implies$ strict H-boundedness $\implies$ $H$-$boundedness$ $\implies$ $M$-$boundedness$. It turns out as the next result manifests that, for the space $C_{m^I}(X)$, these various type of boundedness are equivalent.

\begin{theorem}
	For an ideal $I$ in $C(X)$, the following statements are equivalent:
	\begin{enumerate}
		\item $C_{m^I}(X)$ is hemicompact.
		\item $C_{m^I}(X)$ is $\sigma$-compact.
		\item $C_{m^I}(X)$ is $strictly$ $H$-$bounded$.
		\item $C_{m^I}(X)$ is $H$-$bounded$.
		\item $C_{m^I}(X)$ is $strictly$ $M$-$bounded$.
		\item $C_{m^I}(X)$ is $M$-$bounded$.
		\item  $C_{u^I}(X)$ is $M$-$bounded$.
		\item $X$ is finite and $I=C(X)$.
		\item $C_{m^I}(X)$ is locally compact and $I$ is a free ideal in $C(X)$.
		\item $C_{m^I}(X)$ is locally countably compact and $I$ is a free ideal in $C(X)$.
		\item $C_{m^I}(X)$ is locally pseudocompact and $I$ is a free ideal in $C(X)$.
	\end{enumerate}
\end{theorem}  
\begin{proof}
	The chain implication $(1)\implies(2)\implies(3)\implies(4)\implies(6)\implies(7)$ is immediate. Since $M$-boundedness $\implies\aleph_0$-boundedness, implies in view of Lemma \ref{Lemma} that $I$ is a free ideal in $C(X)$. Since the $u$-topology on $C(X)$ is weaker than the $u^I$-topology, the condition (7) further implies that $C_u(X)$ is $M$-$bounded$. This in turn implies in view of \cite[Theorem 2.8]{KH} that $X$ is finite and therefore a compact space and hence each proper ideal is fixed \cite[Theorem 4.8]{GJ}. Thus $I=C(X)$ and $(7)\implies (8)$ is proved. $(8)\implies (1) $ follows from Theorem 5.2 in \cite{MKJ}. $(8)\implies (9)$ also follows from Theorem 5.2 in \cite{MKJ}. $(9)\implies(10)\implies(11)$ is immediate.\\	
	$(11)\implies(8)$: Suppose $(11)$ is true. Then $I$ being a free ideal in $C(X)$ we have $\bigcap Z[I]=\emptyset$. It follows from  Theorem \ref{c} that $X$ is a finite set.
\end{proof} 
\section{Connectedness of $C_{u^I}(X)$ and $C_{m^I}(X)$}

The component of $\underline{0}$ in $C_{u^I}(X)$ is $C^*(X)\cap I$ \cite[Theorem 3.12]{PA2023} and in $C_{m^I}(X)$ is $C_\psi(X)\cap I$ \cite[Theorem 3.4]{A2017}. In this article we give an alternative proof of these known facts. Furthermore we show that $C^*(X)\cap I$ and $C_\psi(X)\cap I$ are the path components and also the quasi-components of $\underline{0}$ in the respective topologies.
\begin{lemma}\label{8.1}
	$C^*(X)\cap I$ is pathwise connected in $C_{u^I}(X)$.
\end{lemma}
\begin{proof}
	It is enough to show that each $h\in C^*(X)\cap I$ is pathwise connected to $\underline{0}$. Let $h\in C^*(X)\cap I$. Define $p:[0,1]\rightarrow C(X)$ by $p(t)=t \cdot h$ for all $t\in [0,1]$. Note that for any $t\in [0,1]$, $p(t)\in C^*(X)\cap I$, $p(0)=\underline{0}$ and $p(1)=h$. Let $t\in [0,1]$ and $B_u(p(t),I,\epsilon)$ be any basic open set containing $p(t)$ in $C_{u^I}(X)$. Since $h\in C^*(X)$, there exists an $M>0$ such that $|f(x)|<M$ for all $x\in X$. Take $\delta=\frac{\epsilon}{2M}$. Let $t'\in (t-\delta,t+\delta)\cap[0,1]$ and $x\in X$, we have $|p(t)(x)-p(t')(x)|=|t-t'||h(x)|<M\cdot \delta<\epsilon$ this implies $\sup\limits_{x\in X}|p(t)(x)-p(t')(x)|\leq \frac{\epsilon}{2}<\epsilon$ and also observe that $p(t)-p(t')\in I$. Hence, $p((t-\delta,t+\delta)\cap [0,1])\subseteq B_u(p(t),I,\epsilon) $ and that settles the continuity of $p$ at the point $t$. 
\end{proof}
\begin{theorem}\label{8.2}
	$C^*(X)\cap I$ is the path component, component and quasi-component of $\underline{0}$ in $C_{u^I}(X)$.
\end{theorem}
\begin{proof}
	It follows from Theorem 3.9(ii) in \cite{PA2023} and Lemma \ref{8.1} that $C^*(X)\cap I$ is a clopen, pathwise connected subset of $C_{u^I}(X)$ containing $\underline{0}$. Let $P$, $C $ and $Q$ be the path component, component and quasi-component of $\underline{0}$ in $C_{u^I}(X)$ respectively. Then $C^*(X)\cap I\subseteq P\subseteq C\subseteq Q$. Also $Q$ is the intersection of all clopen subset of $C_{u^I}(X)$ containing $\underline{0}$ implies that $Q\subseteq C^*(X)\cap I$. Hence we have $P=C=Q=C^*(X)\cap I$.
\end{proof} 
	\begin{lemma}\label{8.3}
		$C_\psi(X)\cap I$ is pathwise connected in $C_{m^I}(X)$.
	\end{lemma}
\begin{proof}
	Let $h\in C_\psi(X)\cap I$. Define $q:[0,1]\rightarrow C(X)$ by $q(t)=t \cdot h$. Note that for any $t\in [0,1], q(t)\in C_\psi(X)\cap I$, $q(0)=0$ and $q(1)=h$. Let $t\in [0,1]$ and $B_m(q(t),I,u)$ be any basic open set containing $q(t)$ in $C_{m^I}(X)$. Since $h\in C_\psi(X)$, there exists an $M>0$, $\epsilon> 0$ such that $|h(x)|<M$ for all $x\in X$ and $\epsilon<u(x)$ for all $x\in cl_X(X\setminus Z(h))$. Set $\delta=\frac{\epsilon}{2M}$. Now for any $t'\in [0,1]\cap(t-\delta,t+\delta)$, we have $$|q(t)(x)-q(t')(x)|=|t-t'||h(x)|$$
	\hspace{8.4cm}$<\delta M$
	
	$\hspace{7.8cm}<\epsilon \text{ for all } x\in X$

	\hspace{7.8cm}$<u(x)$ for all $x\in \cl_X(X\setminus Z(h))$. \newline We further note that the above inequality is valid for each $x\in X\setminus \cl_X(X\setminus Z(f))$ because $h(x)=0$ and hence $q(t)(x)=q(t')(x)=0$. Also note that $q(t)-q(t')\in I$. Therefore, $q(t')\in B_m(q(t),I,u)$. Hence $q$ is continuous and $C_\psi(X)\cap I$ is path connected.
\end{proof}
\begin{lemma}\cite[Lemma 4.1]{MKJ}
	 For any $f\in C(X)$, $A_f=\{g\in C(X):f\cdot g\in C^*(X)\}$ is a clopen subset of $C_m(X)$ and hence clopen in $C_{m^I}(X)$, because the $m^I$-topology on $C(X)$ is finer than the $m$-topology.
	\end{lemma} 
\begin{theorem}\label{8.5}
	The path component, component and quasi-component of \underline{0} in $C_{m^I}(X)$ are equal to $C_\psi(X)\cap I$. 
\end{theorem}
\begin{proof}
 Let $P$, $C $ and $Q$ be the path component, component and the quasi-component of $\underline{0}$ in $C_{m^I}(X)$ respectively. Then from Lemma \ref{8.3}, $C_\psi(X)\cap I\subseteq P\subseteq C\subseteq Q$. Also $Q$ is the intersection of all clopen subset of $C_{m^I}(X)$ containing $\underline{0}$, implies $Q\subseteq \bigcap\limits_{g\in C(X)}A_g\cap I=C_\psi(X)\cap I$, here we use the fact that $I$ is clopen in $C_{m^I}(X)$ \cite[Proposition 2.2]{A2017} and a standard result $C_\psi(X)=\{f\in C(X):f\cdot g\in C^*(X) \text{ for all } g\in C(X)\}$. Hence we have $P=C=Q=C_\psi(X)\cap I$.
\end{proof}

\begin{theorem}
	The following conditions are equivalent.
\begin{enumerate}
		\item $C_{m^I}(X)$ is pathwise connected.
	\item $C_{m^I}(X)$ is connected.
			\item $C_{u^I}(X)$ is pathwise connected.
					\item $C_{u^I}(X)$ is connected.
	\item $I=C(X)$ and $X$ is pseudocompact.
	\item $C_{u^I}(X)$ is a topological vector space.
\end{enumerate}
	\end{theorem}
\begin{proof}
	Follows on combining Theorem \ref*{8.2}, Theorem \ref{8.5} and Theorem 3.4 in \cite{PA2023}.
\end{proof}
Next we show that $C_{m^I}(X)$ is extremally disconnected if and only if $I=\{\underline{0}\}$. Recall that a topological space $Z$ is called extremally disconnected if the closure of every open set is open in $Z$.
\begin{theorem}\label{Ex}
	Let $c\in X\setminus \bigcap Z[I]$. Then the following relations hold:
	\begin{enumerate}
		\item $c^+=\{f\in C(X):f(c)>0\}$ is open in $C_{u^I}(X)$.
		\item $c^0=\{f\in C(X):f(c)\geq 0\}$ is closed in $C_{u^I}(X)$.
		\item $c^0$ is not open in $C_{u^I}(X)$.
		\item $\cl _{u^I} c ^+=c^0$.
	\end{enumerate}
\end{theorem}
\begin{proof}
\	\begin{enumerate}
		\item For any $f\in c^+$, we have $B_u(f,I,\frac{f(c)}{2})\subseteq c^+$ and hence $c^+$ is open in $C_{u^I}(X)$.
		\item Let $h\notin c^0$, then $h(c)<0$. It is easy to verify that $B_u(h,I,\frac{|h(c)|}{2})\cap c^0=\emptyset$. Therefore $c^0$ is closed in $C_{u^I}(X)$.
		\item To show that $c^0$ is not open in $C_{u^I}(X)\cap I$, it is enough to show $\underline{0} $ is not an interior point of $c^0$. If possible, let $\underline{0}$ be an interior point of $c^0$. Then there is a $\epsilon>0$ such that $B_u(\underline{0},I,\epsilon)\subseteq c^0$. Now by Lemma \ref{Lemma}, there is an $h\in I$ such that $|h|\leq1$ and $h(c)=\frac{1}{2}$. Therefore $-\frac{1}{2}h \epsilon\in B_u(\underline{0},I,\epsilon) $ but $-\frac{1}{2}h\epsilon (c)=-\frac{1}{4}\epsilon<0$, which implies that $-\frac{1}{2}h\notin c^0$, a contradiction. 
		\item It suffices to show $c^0\subseteq \cl_{u^I}c^+$. Let $f\in c^0$ and $B_u(f,I,\epsilon)$ be any basic open set containing $f$. By Lemma \ref{Lemma}, there is an $h\in I$ such that $|h|\leq1$ and $h(c)=\frac{1}{2}$. So $f+h\cdot \frac{\epsilon}{2}\in B_u(f,I,\epsilon)\cap c^+$. Consequently, $f\in \cl_{u^I}c^+$.
	\end{enumerate}
\end{proof}
The $m$-analogue of the above result is also true.
\begin{theorem}
	Let $c\in X\setminus \bigcap Z[I]$, then the following relations hold:
		\begin{enumerate}\label{ex}
		\item $c^+=\{f\in C(X):f(c)>0\}$ is open in $C_{m^I}(X)$.
		\item $c^0=\{f\in C(X):f(c)\geq 0\}$ is closed in $C_{m^I}(X)$.
		\item $c^0$ is not open in $C_{m^I}(X)$.
		\item $\cl _{m^I} c ^+=c^0$.
	\end{enumerate}
\end{theorem}

\begin{theorem}
	The following statements are equivalent.
	\begin{enumerate}
		\item $C_{u^I}(X)$ is the discrete space on $C(X)$.
		\item $C_{m^I}(X)$ is the discrete space on $C(X)$.
		\item $C_{u^I}(X)$ is extremally disconnected.
		\item $C_{m^I}(X)$ is extremally disconnected.
		\item $I=\{\underline{0}\}$.
	
		\item $C_{m^I}(X)$ is a $P$-space.
		\item $C_{u^I}(X)$ is a $P$-space.
	\end{enumerate}
\end{theorem}
\begin{proof}
	$(5)\Rightarrow(1)\Rightarrow(2)\Rightarrow (4)$ and $(5)\Rightarrow (1)\Rightarrow(3)$ are immediate.
	
	$(3)\Rightarrow (5)$: It is easy to check $I=\{\underline{0}\}$ if and only if  $\bigcap Z[I]=X$. If possible, let $c\in X\setminus \bigcap Z[I]$. Then by Theorem \ref{Ex}, $\cl_{u^I} c^+$ is not open, which is a contradiction to the assumed extremal disconnectedness of $C_{u^I}(X)$.
	
	$(4)\Rightarrow (5)$: Analogues as above.

	$(7)\Rightarrow(1)$: For any $f\in C(X)$, $\bigcap\limits_{n\in \mathbb{N}} B_u(f,I,\frac{1}{n})=\{f\}$, which is open in $C_{u^I}(X)$. Hence, $C_{u^I}(X)$ is discrete.
	
	$(6)\Rightarrow (2)$: Analogous as above.
	\end{proof}

\section{Completeness conditions on the space $C_{m^I}(X)$}
 We begin this section by proving the following theorem which we feel, is of independent interest.\begin{theorem}\label{Baire}
 	For any ideal $I$ in $C(X)$, if $C_{m^I}(X)$ is a hereditary Baire space, then $X$ is $I$-$pseudocompact$.
 \end{theorem}
\begin{proof}
 Assume that $X$ is not $I$-$pseudocompact$. Then $X\setminus \bigcap Z[I]$ contains a copy of $\mathbb{N}$, $D=\{x_1,x_2\cdots\}$, C-embedded in $X$. Define for each $n\in \mathbb{N}$, $$H_n=\{f\in C(X):f(x_k)=0\text{ for all }k\geq n\}$$ and $H=\bigcup\limits_{n=1}^\infty H_n$. We first show that each $H_k$ is closed in $C(X)$ in the $m^I$-topology. Choose $f\in C(X)$ such that $f\notin H_k$. Then $f(x_p)\neq 0$ for some $p\geq k$. Let $\epsilon(x)=$ the constant function $\frac{1}{2}|f(x_p)|$. We claim that $B_m(f,I,\epsilon)\cap H_k=\emptyset$ and hence $H_k$ is closed in $C_{m^I}(X)$. Towards the proof of last claim we observe that if $g\in B_m(f,I,\epsilon)$, then $|g(x_p)-f(x_p)|<\epsilon(x_p)=\frac{1}{2}|f(x_p)|,$ which implies $g(x_p)\neq0$ and hence $g\notin H_k$. We next assert that the interior of $H_k$ in the space $C_{m^I}(X) $ is void. If possible, let this interior be non void. Then there exists $u\in C_+(X)$ such that $B_m(g,I,u)\subseteq H_k$ for some $g\in H_k$. Now $g\in H_k$ implies that $g(x_{k+1})=0$. On the other hand $x_{k+1}\in X \setminus \bigcap Z[I]$ implies in view of Lemma \ref{Lemma} that there exists $h\in I$ such that $|h|\leq 1$ on $X$ and $h(x_{k+1})=\frac{1}{2}$. This implies that $g+\frac{1}{2}h\cdot u\notin H_k$. But $g+\frac{1}{2}hu-g=\frac{1}{2}hu\in I$ and $\frac{1}{2}h(x) u(x)<u(x)$ for each $x\in X$, consequently $g+\frac{1}{2}hu\in B_m(g,I,u),$ a contradiction. Thus it is settled that each $H_k$ has void interior in the space $C_{m^I}(X)$. To complete the present  implication relation it suffices to show that $H$ is closed in $C_{m^I}(X)$. Let $l\in C(X)\setminus H$. Then for each $n\in \mathbb{N}$, $l\notin H_n$. Therefore there exists a countably infinite subset $D_0=\{x_{n_1},x_{n_2},\cdots\} $ of $D$ such that $l(x_{n_i})\neq 0$ for each $i\in \mathbb{N}$. Since $D_0$ is C-embedded in $X$,there exists $u\in C_+(X)$ such that $u(x_{n_i})=\frac{1}{2}|l(x_{n_i})|$ for each $i\in \mathbb{N}$. We assert that $B_m(l,I,u)\cap H=\emptyset$ and hence $H$ is closed in $C_{m^I}(X)$. Towards the assertion we note that $h\in B_m(l,I,u)$ implies that $|h(x_{n_i})-l(x_{n_i})|<\frac{1}{2}|l(x_{n_i})| $ for each $i\in \mathbb{N}$ and therefore $h(x_{n_i})\neq0$ for each $i\in \mathbb{N}$ and hence $h\notin H$.
\end{proof}
 So far, the ideal $I$ in $C(X)$ considered throughout this article is absolutely arbitrary. But this time in the remaining portion of the final section, some natural condition has to be imposed on $I$ for the equivalence of several statements related to completeness of the space $C_{m^I}(X)$.
\begin{theorem}\label{complete}
	 The metric $d$ on $C(X)$ considered in Theorem \ref{m} is complete if and only if $I$ is closed under uniform limits.
\end{theorem}
\begin{proof}
	First let $I$ be closed under uniform limits and $\{f_n\}$ be a Cauchy sequence of $(C(X),d)$. Then for any $0<\epsilon<1$, there exists $k\in \mathbb{N}$ such that $d(f_n,f_m)<\epsilon$ for all $n,m\geq k$. This yields for $m,n\geq k$, $\sup\limits_{x\in X}|f_m(x)-f_n(x)|<\epsilon$  and $f_m-f_n\in I$. This clearly shows that for each $x\in X$, $\{f_n(x)\}$ is a Cauchy sequence in $\mathbb{R}$ and therefore $\lim\limits_{n\rightarrow\infty} f_n(x)$ exists, and equals to $f(x)$ say for each $x\in \mathbb{R} $. Now since $|f_m(x)-f_n(x)|<\epsilon$ for all $m,n\geq k$ and for all $x\in X$, on taking limit as $n\rightarrow\infty$, this yields $|f(x)-f_n(x)|\leq \epsilon$ for all $n\geq k$ for all $x\in X$ and hence $\sup\limits_{x\in X}|f(x)-f_n(x)|\leq \epsilon$ for all $n\geq k$. Thus $\{f_n\}$ converges to $f$ uniformly on $X$ and so $f\in C(X)$. Since $f_n-f_m\in I$ for each $n,m\geq k$ and $I$ is closed under uniform limit, this clearly implies that $f_n-f\in I$ for all $n\geq k$. Then we can write $d(f_n,f)\leq\epsilon$ for all $n\geq k$. Hence, $d$ is a complete metric on $C(X)$.
	
	To prove the converse part let $d$ be a complete metric on $C(X)$ yet be not closed under uniform limits. This means that there exists a sequence $\{f_n\}$ in $I$, which converges uniformly to some $f \in C(X)$ but $f\notin I$. Therefore, given $0<\epsilon<1$, there exists $k\in \mathbb{N}$ such that $\sup\limits_{x\in X}|f_n(x)-f(x)|<\frac{1}{2}\epsilon$ for all $n\geq k$, this implies that for all $m,n\geq k$, $\sup\limits_{x\in X}|f_n(x)-f_m(x)|\leq\epsilon$. Since $f_m-f_n\in I $ for each $m,n\in \mathbb{N}$, this further implies that $d(f_m,f_n)\leq \epsilon$ for all $n\geq k$. Thus $\{f_m\}$ is a Cauchy sequence in the metric space $(C(X),d)$. Since the metric $d$ already induces the $u^I$-topology on $C(X)$[ vide Theorem \ref{m}], to get the desired contradiction, we shall show that $\{f_n\}$ does not converge to any limit $f$ in the space $C_{u^I}(X)$. As, $f_n\in I$ for each $n\in \mathbb{N}$, while $f\notin I$ and $I$ is an ideal in $C(X)$, this implies that $f_n-f \notin I$ for each $n\in \mathbb{N}$. Consequently, $d(f_n,f)=1$ for each $n\in \mathbb{N}$. Therefore, $f$ is not the limit of the sequence $\{f_n\}$ in $C_{u^I}(X)$. Again if some $g(\neq f)$ in $C(X)$ happen to be the limit of $\{f_n\}$ in the $u^I$-topology on $C(X)$, then since the  $u^I$-topology on $C(X)$ is finer than the topology of uniform convergence on $C(X)$, this implies that $\{f_n\}$ converges to $g$ uniformly over $X$ and hence $f=g$, a contradiction. Thus in any case $\{f_n\}$ does not converge to any limit $f$ in $C_{u^I}(X)$.
\end{proof}
 We feel, the next theorem is also of independent interest. \begin{theorem}
 	Suppose the ideal $I$ in $C(X)$ is closed under uniform limits. Then $C_{m^I}(X)$ is a Baire space.
 \end{theorem} 
 \begin{proof}
 	Let $\{U_n:n\in \mathbb{N}\}$ be a countable family of dense open subsets of $C_{m^I}(X)$. We need to show that $\bigcap\limits_{n=1}^\infty U_n$ is dense in $C_{m^I}(X)$. Let $f\in C(X)$ and $v\in C_+(X)$ . Since $U$ is dense in $C_{m^I}(X)$, then there exists $f_1\in B_m(f,I,v)\cap U_1$. Now we can find out $v'\in C_+(X)$ such that $B_m(f_1,I,v')\subseteq B_m(f,I,v)\cap U_1$. Let $v_1(x)=\min\{v(x),v'(x),1\}$ for all $x\in X$. Then $B_m(f_1,I,v_1)\subseteq B_m(f,I,v)\cap U_1$. In the next step there exists an $f_2\in B_m(f_1,I,\frac{v_1}{3})\cap U_2$. So by induction, for each $n\in \mathbb{N}$ we can find out $f_n\in C(X)$ and $u_n\in C_+(X)$ such that, $B_m(f_n,I,v_n)\subseteq B_m(f_{n-1},I,\frac{v_{n-1}}{3})\cap U_n$ and $v_n\leq \frac{1}{3}v_{n-1}$. Consequently, $\bigcap\limits_{n=1}^\infty B_m(f_n,I,v_n)\subseteq B_m(f,I,v)\cap \bigcap\limits_{n=1}^\infty U_n$. Now it is sufficient to show that $\bigcap\limits_{n=1}^\infty B_m(f_n,I,v_n)$ is non empty. Let $\epsilon>0$, $x\in X$ and $k_1$ be a positive integer such that $\frac{1}{3^{k_1-1}}<\epsilon$. Then for all $n\geq m\geq k_1$, $|f_n(x)-f_m(x)|<v_m(x)<\frac{1}{3^{m-1}}<\frac{1}{3^{k_1-1}}<\epsilon$, for all $x\in X$. So $\{f_n\}$ is a Cauchy sequence in $C_{u^I}(X)$ and hence $\{f_n\}$ converges uniformly to some $h\in C(X)$, because $C_{u^I}(X)$ is a complete metrizable space by Theorem \ref{complete}. We show that $h\in \bigcap\limits_{n=1}^\infty B_m(f_n,I,v_n)$. If possible, let there exist a $k\in\mathbb{N}$ such that $h\notin B_m(f_k,I,v_k)$.  First we show that $f_k-h \in I$. Since $B_m(h,I,v_k)$ is an open set containing $h$, which is the uniform limit of the sequence $\{f_n\}$, then there exits an $m\in \mathbb{N}$ such that $f_n\in B_m(h,I,v_k)$ for all $n\geq m$. If $k\geq m$, it is immediate that $f_k-h\in I$. Now let $k<m$ then $f_m\in B_m(h,I,v_k)$ and hence $f_m-g\in I $. Since $f_m\in B_m(f_k,I,v_k)$ implies $f_k-f_m\in I$. Therefore $f_k-h\in I$. Now as $h\notin B_m(f_,I,v_k)$, then there exists a $c\in X\setminus \bigcap Z[I]$ such that $|f_k(c)-h(c)|\geq v_k(c)$. But $\{f_n(c)\}$ is a sequence of real number which converges to $h(c)$. So there exists a natural number $m\geq k$ such that $|f_n(c)-h(c)|<v_{k+1}(c)$ for all $n\geq m$. In particular $|f_m(c)-h(c)|<v_{k+1}(c)<\frac{1}{3}v_k(c)$. Therefore $|f_k(c)-h(c)|\leq |f_k(c)-f_m(c)|+|f_m(c)-h(c)|<\frac{1}{3}v_k(c)+\frac{1}{3}v_k(c)<v_k(c)$ which is a contradiction.
 \end{proof}
The next theorem tells amongst others that with the additional hypothesis, imposed on $I$, the hereditary Baireness of $C_{m^I}(X)$ is equivalent to the $I$-$pseudocompactness$ of $X$.  
\begin{theorem} For an ideal $I$ in $C(X)$ which is closed under uniform limits, the following statements are equivalent:\begin{enumerate}
		\item The space $C_{m^I}(X)$ is completely metrizable.
		\item $C_{m^I}(X)$ is \v{C}ech-complete.
		\item $C_{m^I}(X)$ is hereditarily Baire.
		\item $X$ is $I$-pseudocompact.
		
		\item  $C_{m^I}(X)$ is locally \v{C}ech-complete.
		
		\item   $C_{m^I}(X)$ is an open continuous image of a paracompact \v{C}ech-complete space.
		\item  $C_{m^I}(X)$ is an open continuous image of a \v{C}ech-complete space. 
		
	\end{enumerate}
\end{theorem}
\begin{proof}
The implication $(3)\Rightarrow (4)$ is already proved in Theorem \ref{Baire}. The remaining implications of this theorem can be proved on using Theorem \ref{Ipsu}, Theorem \ref{Baire} and following closely the arguments in the proof of \cite[Theorem 2.10]{MKJ}.
\end{proof}

\begin{center}
	\large
	\textbf{Declarations}
\end{center}

\textbf{Author’s Contribution:} All authors have equal contribution.

\textbf{Ethical Approval:} Not applicable.

\textbf{Competing interests:} Not applicable.

\textbf{Funding:}  The first author is immensely grateful for the award of research fellowship
provided by the University Grants Commission, New Delhi (NTA Ref. No.
211610214962).

\textbf{Availability of data and materials:} Not applicable.

\end{document}